\documentclass[a4paper,11pt]{amsart}

\usepackage[english]{babel}
\usepackage[latin1]{inputenc}
\usepackage[T1]{fontenc}
\usepackage{fancyhdr}
\usepackage{indentfirst}
\usepackage{amssymb}
\usepackage{amsmath}
\usepackage{amsthm}
\usepackage{mathrsfs}
\usepackage{stmaryrd}
\usepackage{latexsym}
\usepackage{setspace}
\usepackage{mathtools}
\usepackage{etoolbox}
\usepackage{lipsum}

\usepackage[active]{srcltx}
\usepackage{graphicx}
\usepackage{amsmath,amssymb,amsthm,amsfonts}
\usepackage{amssymb}
\numberwithin{equation}{section}

\newtheorem{thm}{Theorem}[section]
\newtheorem{cor}[thm]{Corollary}

\newtheorem{prop}[thm]{Proposition}
\newtheorem{rem}[thm]{Remark}

\newcommand{\bremark}{\begin{rem} \textup}
\newcommand{\eremark}{\end{rem} }

\newcommand{\cuad}{{\sqcap\kern-.68em\sqcup}}

\newcommand{\R}{{\mathbb{R}}}

\newcommand{\KK}{{\mathcal {K}}}

\renewcommand{\rho}{\varrho}
\renewcommand{\theta}{\vartheta}

\begin{document}

\title{BLOW UP OF SOLUTIONS OF SEMILINEAR HEAT EQUATIONS IN GENERAL DOMAINS}

\def\shorttitle{Blow up in general domains}

\author{Valeria Marino}
\address{Dipartimento di Matematica, Universit\`a di Roma "La Sapienza"\\
P.le A. Moro 2\\
00185 Roma, Italy}
\email{valeria.marino.m@gmail.com}

\author{Filomena Pacella}
\address{Dipartimento di Matematica, Universit\`a di Roma "La Sapienza"\\
P.le A. Moro 2\\
00185 Roma, Italy}
\email{pacella@mat.uniroma1.it}
\thanks{F. P.  was partially supported by PRIN 2009-WRJ3W7 grant(Italy).}

\author{Berardino Sciunzi}
\address{Dipartimento di Matematica, UNICAL\\
Ponte Pietro  Bucci 31B \\
87036 Arcavacata di Rende, Cosenza, Italy}
\email{ sciunzi@mat.unical.it}
\thanks{B. S.  was partially supported by ERC-2011-grant: \emph{Epsilon} and PRIN-2011: {\em Var. and Top. Met.}}


\begin{abstract}
Consider the nonlinear heat equation
$v_t -\Delta  v= |v|^{p-1} v$
 in a bounded smooth domain $\Omega\subset \R^n$ with $n>2$ and Dirichlet boundary condition. Given  $u_{p}$   a sign-changing stationary solution fulfilling suitable assumptions, we prove that the solution  with initial value $\theta u_{p} $ blows up in finite time  if $ |\theta -1|>0$ is sufficiently small and if $p$ is sufficiently close to the critical exponent.\\Since for $\theta=1$ the solution is global, this shows that, in general, the set of the initial data for which the solution is global is not star-shaped. This phenomenon had been previously observed in the case when the domain is a ball and the stationary solution is radially symmetric.
 \noindent
\end{abstract}

\maketitle

\section{Introduction}
We consider a nonlinear heat equation of the type
\begin{equation}
\left\{\begin{array}{ll} v_t -\triangle v=|v|^{p-1}v \qquad & \mbox{in }\Omega\times(0,T)\\
                         v=0  \qquad & \mbox{on }\partial\Omega\times(0,T)\\
                         v(0)=v_0(x)  \qquad & \mbox{in }\Omega
       \end{array}\right.\label{eq:probl_parabolico}
\end{equation}
where $\Omega\subset\mathbb{R}^n$, $n\in\mathbb{N}$, is a bounded domain, $p>1$, $T\in(0,+\infty]$ and
\begin{equation*}
v_0\in C_0(\Omega)=\{v\in C(\overline{\Omega}),v(x)=0 \mbox{ for } x\in\partial\Omega\}.
\end{equation*}
It is well known that the initial value problem (\ref{eq:probl_parabolico}) is locally well posed in $C_0(\Omega)$.
Denoting with $T_{v_0}$ the maximal existence time of the solution of (\ref{eq:probl_parabolico}) with initial datum $v_0$, we consider the set of the initial data for which the corresponding solution is global, namely:
\begin{equation}\nonumber
\mathcal{G}=\{v_0\in C_0(\Omega), T_{v_0}=\infty\}\,.
\end{equation}
It is interesting to understand the geometrical properties of the set $\mathcal{G}$.
If we consider $v_0=\theta w$, with $w\in C_0(\Omega)$ and $\theta\in\mathbb{R}$, it is well known that if $|\theta|$ is small enough the solution of $(\ref{eq:probl_parabolico})$ with initial datum $\theta w$, exists globally. Moreover, if $|\theta|$ is sufficiently large, it is easy to see that  the solution blows up in finite time as a consequence of the fact that it has negative energy (see \cite{Levine} and \cite{Ball}).
It is interesting to understand what happens for intermediate values of $\theta$. The case when $w$ is positive is completely clear, as a matter of fact from the maximum principle for the heat equation it follows that there exists $\widetilde\theta>0$ such that if $0<\theta<\widetilde\theta$ then the solution with initial value $\theta\,w$ is globally defined, while if $\theta>\widetilde\theta$ it blows up in finite time. In the borderline case both global existence or  blow up in finite time can occur.

Thus, if we define $\mathcal{G}^+=\{v_0\in\mathcal{G}, v_0\geq 0\}$, we can assert that $\mathcal{G}^+$ is star-shaped with respect to 0 (indeed it is a convex set).
When the initial value changes sign the situation is different and, in general, the set $\mathcal{G}$ may be not star-shaped.
In fact, if we define by $u_p$ a radial sign changing solution of the stationary problem
\begin{equation}
\left\{\begin{array}{ll} -\triangle u_p=|u_p|^{p-1}u_p\qquad & \mbox{in }\Omega\\
                         u_p=0 \qquad & \mbox{on }\partial\Omega\\
       \end{array}\right.\label{eq:probl}
\end{equation}
where $\Omega$ is the unit ball in $\mathbb{R}^n$, with $n>2$ and $p>1$, it has been shown in \cite{C-D-W} that there exists $p^*<p_S$, with $p_S=\frac{n+2}{n-2}$ and there exists $\epsilon >0$ such that if $p^*<p<p_S$ and $0<|1-\theta|<\epsilon$ then $\theta u_p\not\in\mathcal{G}$ i.e. the solution of (\ref{eq:probl_parabolico}), with initial datum $\theta u_p$, blows up in finite time both for $\theta$ slightly greater and slightly smaller than 1.
Hence $\mathcal{G}$ is not star-shaped since $u_p\in\mathcal{G}$.\\
\noindent Recently a similar result has been proved in \cite{D-P-S} in the case when the dimension is two and the exponent $p$ is sufficiently large.

Such a result does not hold in the case $n=1$ (always considering $p>1$).
 As a matter of fact in the one-dimensional case we have that for $|\theta|<1$, $v_{\theta,p}$ (the solution with initial value $\theta\,u_p$) is global and converges uniformly to zero, while it blows up in finite time if $|\theta|>1$.

The proofs of the results of \cite{C-D-W} and \cite{D-P-S} exploit strongly the radial symmetry of the stationary solutions. Hence it is natural to ask whether a similar result holds also in general domains and what kind of sign changing stationary solutions give rise to this phenomenon. Note that this  cannot be true for any
 sign changing stationary solution as it is easy to see considering, for example, a nodal solution in the ball which
 is odd with respect to a symmetry hyperplane and has only two nodal domains.\\
 \noindent Here we show that, in the case when $n>2$ and for exponents close to the critical one, the same blow up phenomenon occurs in any bounded domain  considering a suitable class of  sign changing  solutions $u_p$ of \eqref{eq:probl}.\\
\noindent More precisely we deal with solutions $u_p$ of (\ref{eq:probl}) with the following properties:
\begin{itemize}\label{hp}
\item[(a)]\,\,$\int_\Omega|\triangledown u_p|^2dx\shortrightarrow 2S^{\frac{n}{2}}\qquad$ as $p\shortrightarrow p_S$,
\item[(b)]\,\,$\frac{\max u_p}{\min u_p}\shortrightarrow -\infty\qquad\qquad\,\,\,$ as $p\shortrightarrow p_S$,
\end{itemize}
where $S$ is the best Sobolev constant for the embedding of $H_0^1(\Omega)$ into $L^{2^*}(\Omega)$.
 It has been proved in \cite{P-W} that such solutions  exist,  assuming that $\Omega$ is a bounded smooth domain in $\mathbb{R}^n$ with $n>2$, symmetric with respect to the $x_i$-coordinates ($i=1,\ldots,n$). Later in \cite{M-P} the authors extend the same result to any general bounded and smooth domain in $\mathbb{R}^n$, with $n>2$.
Moreover in  \cite{B-E-P_2007} it has been proved that condition $(a)$ implies that $\Omega\backslash\{x\in\Omega\,|\,u_p(x)=0\}$ has exactly two connected components while, when $n\geq 4$, $(b)$ implies that the nodal surface of $u_p$ does not intersect the boundary $\partial\Omega$ and the positive part  $u^+_p$ and the negative part $u^-_p$ concentrate at at the same point.
 One could easily verify that $(a)$ is equivalent to
\begin{equation*}
E_p(u_p)=\frac{1}{2}\int_\Omega|\triangledown u_p|^2dx-\frac{1}{p+1}\int_\Omega|u_p|^{p+1}dx\shortrightarrow\frac{2}{n}S^{\frac{n}{2}}\qquad\mbox{ as }p\shortrightarrow p_S.\nonumber
\end{equation*}
We refer to \cite{B-E-P_2007} for further properties of such solutions.\\

Our goal is to prove the following theorem.

\newtheorem{theorem}{Theorem}[section]
\begin{theorem}\label{teorema1}
Given  problem (\ref{eq:probl_parabolico}) with $n>2$, $1<p<p_S=\frac{n+2}{n-2}$, and $\Omega$ a bounded smooth domain in $\mathbb{R}^n$, there exists $p^*<p_S$ with the following property: \\ \noindent if $p^*<p<p_S$ and $u_p$ is a sign changing solution of the stationary problem (\ref{eq:probl}) satisfying $(a)$ and $(b)$ then there exist $0<\underline\theta<1<\overline\theta$ such that if $\underline\theta<\theta<\overline\theta$ and $\theta\neq 1$ then $v_{\theta,p}$, solution of (\ref{eq:probl_parabolico}) with initial value $\theta u_p$, blows up in finite time.
\end{theorem}

\noindent To prove Theorem \ref{teorema1} we use the following result which has been proved in \cite{C-D-W} for general domains.

\begin{prop} \label{base}
Let $u_p$ be a sign changing solution of~\eqref{eq:probl} and let
$\varphi_{1,p} $ be a first  eigenfunction of the linearized operator $L_p$  at $u_p$. Assume that
\begin{equation}\nonumber
\int _\Omega u_p\,  \varphi_{1,p} \not= 0.
\end{equation}
Then there exists $\varepsilon >0$ such that if $0<|1-\theta |<\varepsilon $, then $v_{\theta,p}$, solution of (\ref{eq:probl_parabolico}) with initial value $\theta u_p$, blows up in finite time.
\end{prop}

\noindent Thus Theorem \ref{teorema1} will be a consequence of the following

\begin{theorem}\label{teo:integrale}
Let  $n> 2$, $1<p<p_S$, $\Omega\subset\mathbb{R}^n$ a bounded smooth domain and $u_p$ a sign changing solution of (\ref{eq:probl}) satisfying conditions $(a)$ and $(b)$.
Then there exists $p^*<p_S$ such that for $p^*<p<p_S$
\begin{equation}\label{jkgkdfhgl}
\int_\Omega u_p\varphi_{1,p}\, dx>0,
\end{equation}
where $\varphi_{1,p}$ is the first positive eigenfunction of the linearized operator $L_p$  at $u_p$.
\end{theorem}

\noindent
Let us point out that for the proof of Theorem \ref{teo:integrale} the property $(b)$  of our stationary solutions is crucial.
Note that both  properties $(a)$ and $(b)$ are actually satisfied in the special case of  radial sign changing solutions of  (\ref{eq:probl}) (in the ball) with two nodal regions.\\
\noindent So this clarifies that it is neither the symmetry nor the one-dimensional character of the solution which leads to the blow up result obtained in \cite{C-D-W} but rather these properties of the stationary solution that can hold in any bounded domain. Therefore we believe that  also for other semilinear problems where such solutions exist, the same blow up result should be true.\\

\noindent The proof of Theorem \ref{teo:integrale} is based on a rescaling argument about the maximum point of $u_p$.
Indeed, analyzing the asymptotic behavior of the rescaled solutions and of the rescaled first eigenfunctions, we are able to prove \eqref{jkgkdfhgl} by using the properties of the solutions of the limit problem.\\
\noindent The same result of Theorem \ref{teo:integrale} can be  easily extended to the case when the initial datum is a nodal solution $u_{p,\KK}$ of \eqref{eq:probl} with a fixed number $\KK>2$ of nodal regions satisfying:
\begin{itemize}\label{hpuhihg}
\item[$(a)_{\KK}$]\,\,$\int_\Omega|\triangledown u_{p,\KK}|^2dx\leq C$,
\item[$(b)_{\KK}$] $\exists$ a nodal region $\Omega_p^1$ such that, setting
\begin{center}
$u^1_{p,\KK}\,:=\,u_{p,\KK}\cdot\chi_{\Omega_p^1}\quad$ and $\quad\hat{u}_{p,\KK}\,:=\,u_{p,\KK}\cdot\chi_{\Omega\setminus\Omega_p^1}$
\end{center}
then
\[
\int_{\Omega_p^1}|\triangledown u^1_{p,\KK}|^2dx\rightarrow S^{\frac{n}{2}}\qquad \text{as}\,\,\,p\shortrightarrow p_S
\]
and
\[
\frac{\|u^1_{p,\KK}\|_\infty}{\|\hat{u}^1_{p,\KK}\|_\infty}\rightarrow \infty \quad\qquad\qquad \text{as}\,\,\,p\shortrightarrow p_S\,.
\]
\end{itemize}
Solutions of this type have been found in \cite{M-P,P-W} but other kind of solutions could be considered.\\

\noindent The outline of the proof is the following. In Section \ref{kdjfkldksf} we prove some preliminary results,
while in Section \ref{kdhfsklhfbvcbxvbxvcbx} we study the asymptotic behavior of the first eigenvalue and of the first eigenfunction of the linearized operator at $u_p$. Finally in Section \ref{kfkhfkshfkshfhsfhsfhksasasa} we prove Theorem \ref{teo:integrale}.

\section{Preliminaries}\label{kdjfkldksf}
Let us start by recalling some  properties of our solutions.
\newtheorem{lemma}[theorem]{Lemma}
\begin{lemma}\label{lemma:hp}
Let $(u_p)$ be a family of sign-changing solutions of (\ref{eq:probl}) satisfying $(a)$. Then
\begin{eqnarray}
(i)& \int_\Omega |\triangledown u_p^+|^2dx\xrightarrow{p\shortrightarrow p_S} S^{\frac{n}{2}},&
\int_\Omega |\triangledown u_p^-|^2dx\xrightarrow{p\shortrightarrow p_S} S^{\frac{n}{2}},\nonumber\\
(ii)& \int_\Omega (u_p^+)^{\frac{2n}{n-2}}dx\xrightarrow{p\shortrightarrow p_S} S^{\frac{n}{2}},&
\int_\Omega (u_p^-)^{\frac{2n}{n-2}}dx\xrightarrow{p\shortrightarrow p_S} S^{\frac{n}{2}},\nonumber\\
(iii)& u_p\rightharpoonup 0\mbox{ as }p\shortrightarrow p_S,  \nonumber\\
(iv)& M_{p,+}:=\max_{\Omega}u_p^+\xrightarrow{p\shortrightarrow p_S} +\infty,& M_{p,-}:=\max_{\Omega}u_p^-\xrightarrow{p\shortrightarrow p_S} +\infty,\nonumber
\end{eqnarray}
with $u_p^+=\max_{\Omega}(u_p,0)$ and $u_p^-=\max_{\Omega}(-u_p,0)$.
\end{lemma}

\proof
We refer the reader to \cite[Lemma 2.1]{B-E-P_2006}.
\endproof
We now describe the rescaled problem.
Let us define
\begin{equation}
\widetilde{u}_p(x):=\frac{1}{M_p}u_p\biggl(a_p+\frac{x}{M_p^{\frac{p-1}{2}}}\biggl),\quad \mbox{ for }x\in\widetilde\Omega_p:=M_p^{\frac{p-1}{2}}\big(\Omega - a_p\big).\label{eq:u_ptilde}
\end{equation}
where $a_p$ and $M_p$ are such that $|u_p(a_p)|=\|u_p\|_{L^\infty(\Omega)}=:M_p$. Without loss of generality, we can assume that $u_p(a_p)>0$.

Let us consider the limit problem in $\mathbb{R}^n$, that is
\begin{equation}
\left\{\begin{array}{ll} -\triangle u=|u|^{p_S-1}u=|u|^{\frac{4}{n-2}}u\qquad & \mbox{in }\mathbb{R}^n\\
                         u(0)=1.
       \end{array}\right.\label{eq:critico}
\end{equation}
It is well known that the unique regular positive  solution is radial and is given by
\begin{equation*}
U(x)=\Bigg(\frac{n(n-2)}{n(n-2)+|x|^2}\Bigg)^{\frac{n-2}{2}}\qquad\qquad\mbox{for }x\in\mathbb{R}^n.
\end{equation*}
Moreover any sign changing solution of \eqref{eq:critico} has energy larger than $2S^{\frac{n}{2}}$.
\noindent We have
\begin{lemma} For $p\shortrightarrow p_S$
\begin{equation}\nonumber
\widetilde{u}_p\longrightarrow U \mbox{ in } C^2_{loc}(\mathbb{R}^n).
\end{equation}
\end{lemma}

\proof
The proof is the same (with obvious changes) as the one of the similar statement in Theorem 1.1 of \cite{B-E-P_2006} (see page 777).
\endproof

Now we study the linearization of the limit problem (\ref{eq:critico}), so we define the operator
\begin{equation}
L^*(v):=-\triangle v-p_S|U|^{p_S-1}v,\qquad v\in H^2(\mathbb{R}^n)\nonumber
\end{equation}
where $U$ is the solution of (\ref{eq:critico}). The Rayleigh functional associated to $L^*$ is
\begin{equation}
\mathcal{R}(v)=\int_{\mathbb{R}^n}|\triangledown v|^2- p_S|U|^{p_S-1}v^2 dx\nonumber
\end{equation}
and we define
\begin{equation}
\lambda_1^*:=\inf_{\begin{subarray}{c} v\in H^1(\mathbb{R}^n),\\
\|v\|_{{L^2}(\mathbb{R}^n)}=1
\end{subarray}}\mathcal{R}(v).\label{eq:inf}
\end{equation}
We observe that $\lambda_1^*>-\infty$, since $U$ is bounded.

\begin{rem}\label{kdhfsfhkdhkdhgkdhg}
It can be shown, with standard arguments, that there exists a unique positive minimizer $\varphi_1^*$ to (\ref{eq:inf}) which is radial and radially nonincreasing; moreover $\lambda_1^*$ is an eigenvalue of $L^*$ and $\varphi_1^*$ is an eigenvector associated to $\lambda_1^*$. For further details see \cite{L-L}.
\end{rem}

\theoremstyle{plain}
\begin{prop}\label{prop:spectral}
We have the following.
\begin{itemize}
\item[(i)]  $\lambda_1^*<0$,
\item[(ii)] every minimizing sequence of (\ref{eq:inf}) has a subsequence which strongly converges in $L^2(\mathbb{R}^n)$.
\end{itemize}
\end{prop}

\proof
Let us compute $\mathcal{R}$ on $U\in H^1(\mathbb{R}^n)$, solution of the limit problem (\ref{eq:critico}). We have
\begin{eqnarray}
\mathcal{R}(U)&=&\int_{\mathbb{R}^n}|\triangledown U|^2- p_S|U|^{p_S+1}dx\nonumber\\
&=&(1-p_S)\int_{\mathbb{R}^n}|\triangledown U|^2dx<0\nonumber
\end{eqnarray}
since $p_S>1$. By definition (\ref{eq:inf}) this implies that $\lambda_1^*<0$.
To prove (ii) let us consider a sequence $w_n\in H^1(\mathbb{R}^n)$, with $\|w_n\|_{L^2(\mathbb{R}^n)}=1$, which minimizes (\ref{eq:inf}). It is easy to see that $w_n$ is bounded in $H^1(\mathbb{R}^n)$; therefore, up to a subsequence, it converges weakly to some $w\in H^1(\mathbb{R}^n)$ and strongly in $L^2(\{|x|\leq R\})$ for every $R>0$.\\
By the lower semicontinuity of the norm we have
\begin{eqnarray*}
\int_{\mathbb{R}^n}|\triangledown w|^2dx\leq \liminf_{n\shortrightarrow\infty}
\int_{\mathbb{R}^n}|\triangledown w_n|^2dx\quad\mbox{ and }\quad
\|w\|_{L^2(\mathbb{R}^n)}\leq 1.
\end{eqnarray*}
Moreover, for every $\epsilon>0$, we have
\begin{eqnarray*}
&&\Bigg|\int_{\mathbb{R}^n}U^{p_S-1}(w_n^2-w^2)dx\Bigg|
\leq\int_{\mathbb{R}^n}U^{p_S-1}|w_n^2-w^2|dx\\
&=&\int_{|x|\leq R}U^{p_S-1}|w_n^2-w^2|dx+\int_{|x|>R}U^{p_S-1}|w_n^2-w^2|dx\\
&\leq& c\int_{|x|\leq R}|w_n^2-w^2|dx+\frac{c}{R^4}\int_{|x|>R}|w_n^2-w^2|dx\\
&\leq& c\|w_n-w\|_{L^2(|x|\leq R)}+\frac{c}{R^4}\leq\epsilon
\end{eqnarray*}
where the last estimate is possible if we fix $R$ large enough and then we take $n$ sufficiently large. Thus
\begin{equation*}
\int_{\mathbb{R}^n}U^{p_S-1}w_n^2dx\shortrightarrow\int_{\mathbb{R}^n}U^{p_S-1}w^2dx,
\qquad\mbox{ as }\quad n\shortrightarrow\infty
\end{equation*}
and so
\begin{eqnarray*}
\mathcal{R}(w)&=&\int_{\mathbb{R}^n}|\triangledown w|^2- p_S|U|^{p_S-1}w^2dx\\
&\leq&\liminf_{n\shortrightarrow\infty}\int_{\mathbb{R}^n}|\triangledown w_n|^2- p_S|U|^{p_S-1}w_n^2dx=\lambda_1^*\,.
\end{eqnarray*}
This implies that $w\not\equiv 0$ and we can define:
\begin{equation*}
\widehat{w}=\frac{w}{\|w\|_{L^2(\mathbb{R}^n)}}\,.
\end{equation*}
If we assume now by contradiction that $\|w\|^2_{L^2(\mathbb{R}^n)}<1$, it follows that
\begin{equation}
\lambda_1^*\leq\mathcal{R}(\widehat{w})=\frac{\mathcal{R}(w)}{\|w\|^2_{L^2(\mathbb{R}^n)}}
\leq\frac{\lambda_1^*}{\|w\|^2_{L^2(\mathbb{R}^n)}}<\lambda_1^*,\label{eq:28}
\end{equation}
since $\lambda_1^*<0$. By (\ref{eq:28}) we deduce therefore that $\|w\|^2_{L^2(\mathbb{R}^n)}=1$ and so $w$ is a minimizer. This also allows us to deduce that $w_n$ strongly converges to $w$ in $L^2(\mathbb{R}^n)$, and this concludes the proof of (ii).
\endproof

\section{Asymptotic spectral analysis}\label{kdhfsklhfbvcbxvbxvcbx}

We  consider the linearized operator  at $u_p$, that is:
\begin{equation}\nonumber
L_p=-\triangle - p |u_p|^{p-1}I.
\end{equation}
We denote by $\lambda_{1,p}$ the first eigenvalue of $L_p$ in $\Omega$ and by $\varphi_{1,p}$ the corresponding positive eigenfunction such that $\varphi_{1,p}>0$ and $\|\varphi_{1,p}\|_{L^2(\Omega)}=1$. We have
\begin{equation}
-\triangle\varphi_{1,p}-p|u_p|^{p-1}\varphi_{1,p}=\lambda_{1,p}\varphi_{1,p}\qquad
\mbox{in }\Omega.\label{eq:varphi}
\end{equation}
Let us define $\widetilde\varphi_{1,p}$ by
\begin{equation}
\widetilde\varphi_{1,p}(x)=\Big(\frac{1}{\scriptstyle{ M_p^{\frac{p-1}{2}}}}\Big)^{\frac{n}{2}}\varphi_{1,p}\Big(a_p+\frac{x}{\scriptstyle{M_p^{\frac{p-1}{2}}}}\Big) \qquad\mbox{ in }\widetilde\Omega_p,\nonumber
\end{equation}
and $\widetilde\varphi_{1,p}=0$ outside $\widetilde\Omega_p$. It is easy to see that   $\|\widetilde\varphi_{1,p}\|_{L^2(\mathbb{R}^n)}=1$ and $\widetilde\varphi_{1,p}$ satisfies
\begin{equation}
-\triangle\widetilde\varphi_{1,p}-V_p\widetilde\varphi_{1,p}=\widetilde\lambda_{1,p}\widetilde\varphi_{1,p}\nonumber
\end{equation}
where
\begin{equation}
V_p(x)=p\frac{1}{M_p^{p-1}}\Big|u_p\Big(a_p+\frac{x}{\scriptstyle {M_p^{\frac{p-1}{2}}}}\Big)\Big|^{p-1}=p|\widetilde u_p(x)|^{p-1}\nonumber
\end{equation}
and
\begin{equation}
\widetilde\lambda_{1,p}=\frac{\lambda_{1,p}}{M_p^{p-1}}.\nonumber
\end{equation}
This means that $\widetilde\varphi_{1,p}$ is a first eigenfunction of the operator
\begin{equation}
\widetilde L_p=-\triangle -p|\widetilde u_p|^{p-1}I\nonumber
\end{equation}
 and $\widetilde\lambda_{1,p}$ is the corresponding first eigenvalue.
\begin{lemma}\label{lemma:bounded}
The set $\{\widetilde\varphi_{1,p},1<p<p_S\}$ is bounded in $H^1(\mathbb{R}^n)$.
\end{lemma}
\proof
As we have already remarked $\|\widetilde\varphi_{1,p}\|_{L^2(\mathbb{R}^n)}=1$.
Moreover, since $\lambda_{1,p}<0$ and  and $p<p_S$, we get
\begin{eqnarray*}
\int_{\mathbb{R}^n}|\triangledown\widetilde\varphi_{1,p}(x)|^2dx&=&
\frac{1}{M_p^{p-1}}\int_{\widetilde\Omega_p}\bigg(\frac{1}{\scriptstyle{M_p^{\frac{p-1}{2}}}}\bigg)^n\bigg|\triangledown\varphi_{1,p}\bigg(a_p+\frac{x}{\scriptstyle{M_p^{\frac{p-1}{2}}}}\bigg)\bigg|^2dx\\
&=&\frac{1}{M_p^{p-1}}\int_{\Omega}|\triangledown\varphi_{1,p}(x)|^2dx\\
&=&\frac{1}{M_p^{p-1}}\int_{\Omega}p|u_p|^{p-1}\varphi_{1,p}^2dx+
\frac{\lambda_{1,p}}{M_p^{p-1}}\int_{\Omega}\varphi_{1,p}^2dx\\
&\leq&\int_{\Omega}p\bigg(\frac{|u_p|}{M_p}\bigg)^{p-1}\varphi_{1,p}^2dx\\
&\leq& p\int_{\Omega}\varphi_{1,p}^2dx<p_S,
\end{eqnarray*}
i.e. the assertion.
\endproof

\begin{theorem}\label{teo:lambda}
We have
\begin{equation}\label{gfgfgbvbvbncncnmxmx}
\widetilde\lambda_{1,p}\shortrightarrow \lambda_1^*\qquad\mbox{ as }\quad
p\shortrightarrow p_S.
\end{equation}
\end{theorem}

\proof
We divide the proof in two steps:
\begin{description}
\item[Step 1.] We show that for $\epsilon>0$ we have
\begin{equation}
\lambda_1^*\leq \widetilde\lambda_{1,p}+\epsilon\quad\mbox{ for $p$ sufficiently close to $p_S$.}\label{eq:step1}
\end{equation}
\noindent By  (\ref{eq:inf}), we have $\lambda_1^*\leq\mathcal{R}(\widetilde\varphi_{1,p})$. Thus
\begin{eqnarray}
\lambda_1^*&\leq&\int_{\mathbb{R}^n}|\triangledown\widetilde\varphi_{1,p}|^2- p_S|U|^{p_S-1}\widetilde\varphi_{1,p}^2dx\nonumber\\
&=&\int_{\widetilde\Omega_p}|\triangledown\widetilde\varphi_{1,p}|^2-p|\widetilde u_p|^{p-1}\widetilde\varphi_{1,p}^2dx-
\int_{\widetilde\Omega_p} \big{(}p_S|U|^{p_S-1}-p|\widetilde u_p|^{p-1}\big)\widetilde\varphi_{1,p}^2dx\nonumber\\
&=&\widetilde\lambda_{1,p}-\int_{\widetilde\Omega_p} \big(p_S|U|^{p_S-1}-p|\widetilde u_p|^{p-1}\big)\widetilde\varphi_{1,p}^2dx\nonumber\\
&=&\widetilde\lambda_{1,p}-\int_{\widetilde\Omega_p\cap|x|\leq R} \big(p_S|U|^{p_S-1}-p|\widetilde u_p|^{p-1}\big)\widetilde\varphi_{1,p}^2dx+\nonumber\\
&&-\int_{\widetilde\Omega_p\cap|x|>R} \big(p_S|U|^{p_S-1}-p|\widetilde u_p|^{p-1}\big)\widetilde\varphi_{1,p}^2dx\nonumber
\end{eqnarray}
where $R>0$.
Let us first consider the last integral. We want to show that it can be made arbitrarily small. We have
\begin{eqnarray}
\Big|\int_{\widetilde\Omega_p\cap|x|>R}p_S|U|^{p_S-1}\widetilde\varphi_{1,p}^2dx\Big|
&\leq& p_S \int_{\widetilde\Omega_p\cap|x|>R}|U|^{p_S-1}\widetilde\varphi_{1,p}^2dx\nonumber\\
&\leq&\frac{C_1}{R^4}\int_{\mathbb{R}^n}\widetilde\varphi_{1,p}^2dx
\leq \frac{C_1}{R^4}\label{eq:15}
\end{eqnarray}
for some constant $C_1>0$. Therefore we can choose $R$ so large that
\begin{equation}\nonumber
\Big|\int_{\widetilde\Omega_p\cap|x|>R}p_S|U|^{p_S-1}\widetilde\varphi_{1,p}^2dx\Big|\leq\epsilon.
\end{equation}
To estimate the term
\begin{equation*}
\Big|\int_{\widetilde\Omega_p\cap|x|>R}p|\widetilde u_p|^{p-1}\widetilde\varphi_{1,p}^2dx\Big|
\end{equation*}
note that we can split the integral on $\widetilde\Omega_p$ in the integral on
\begin{equation}\label{eq:omega+}
\widetilde\Omega_p^+=\{x\in\widetilde\Omega_p: \widetilde u_p(x)\geq 0\}
\end{equation}
and the one on
\begin{equation}\label{eq:omega-}
\widetilde\Omega_p^-=\{x\in\widetilde\Omega_p: \widetilde u_p(x)< 0\}.
\end{equation}
\noindent Therefore  we get
\begin{eqnarray}
&&\Big|\int_{\widetilde\Omega_p\cap|x|>R}p|\widetilde u_p|^{p-1}\widetilde\varphi_{1,p}^2dx\Big|\leq\nonumber\\
&&\int_{\widetilde\Omega_p^+\cap|x|>R}p|\widetilde u_p|^{p-1}\widetilde\varphi_{1,p}^2dx+\int_{\widetilde\Omega^-_p\cap|x|>R}p|\widetilde u_p|^{p-1}\widetilde\varphi_{1,p}^2dx\label{eq:33}.
\end{eqnarray}
As for the first term of (\ref{eq:33}) we have
\begin{eqnarray}
&&\int_{\widetilde\Omega^+_p\cap|x|>R}p|\widetilde u_p|^{p-1}\widetilde\varphi_{1,p}^2dx\label{eq:7} \\
&&\leq p\Big(\int_{\widetilde\Omega^+_p\cap|x|>R}|\widetilde u_p|^{\frac{n(p-1)}{2}}dx\Big)^{\frac{2}{n}}
\Big(\int_{\widetilde\Omega^+_p\cap|x|>R}\widetilde\varphi_{1,p}^{\frac{2n}{n-2}}dx\Big)^{\frac{n-2}{n}}\nonumber\\
&&\leq p\Big(\int_{\widetilde\Omega^+_p\cap|x|>R}|\widetilde u_p|^{\frac{n(p-1)}{2}}dx\Big)^{\frac{2}{n}}
\|\widetilde\varphi_{1,p}\|_{L^{\frac{2n}{n-2}}(\mathbb{R}^n)}^2\nonumber\\
&&\leq C_2 \Big(\int_{\widetilde\Omega^+_p\cap|x|>R}|\widetilde u_p|^{\frac{n(p-1)}{2}}dx\Big)^{\frac{2}{n}}\nonumber
\end{eqnarray}
where we have used H\"{o}lder's inequality (with exponents $\frac{n}{2}$ and $\frac{n}{n-2}$) for the first estimate and the fact that, as a consequence of Lemma \ref{lemma:bounded},
$\widetilde\varphi_{1,p}$ is bounded in $L^{\frac{2n}{n-2}}(\mathbb{R}^n)$  to obtain the last inequality.\\
In order to estimate the last term in (\ref{eq:7}), we use (ii) of Lemma \ref{lemma:hp} to get
\begin{eqnarray*}
\int_{\widetilde\Omega^+_p\cap |x|\leq R}|\widetilde u_p|^{\frac{n(p-1)}{2}}dx&+&
\int_{\widetilde\Omega^+_p\cap |x|>R}|\widetilde u_p|^{\frac{n(p-1)}{2}}dx=\\
=\int_{\widetilde\Omega^+_p}|\widetilde u_p|^{\frac{n(p-1)}{2}}dx&\xrightarrow{p\shortrightarrow p_S}& S^{\frac{n}{2}}=\int_{\mathbb{R}^n}|U|^{\frac{2n}{n-2}}dx=\\
\int_{|x|\leq R}|U|^{\frac{2n}{n-2}}dx&+&
\int_{|x|>R}|U|^{\frac{2n}{n-2}}dx
\end{eqnarray*}
As $\widetilde u_p\xrightarrow{p\shortrightarrow p_S} U$ in $C^2_{loc}(\mathbb{R}^n)$, we have
\begin{equation*}
\int_{\widetilde\Omega^+_p\cap |x|\leq R}|\widetilde u_p|^{\frac{n(p-1)}{2}}dx\xrightarrow{p\shortrightarrow p_S}\int_{|x|\leq R}|U|^{\frac{2n}{n-2}}dx
\end{equation*}
and so
\begin{equation}
\int_{\widetilde\Omega^+_p\cap |x|> R}|\widetilde u_p|^{\frac{n(p-1)}{2}}dx\xrightarrow{p\shortrightarrow p_S}\int_{|x|> R}|U|^{\frac{2n}{n-2}}dx\label{eq:4}
\end{equation}
but, as $U\in L^{\frac{2n}{n-2}}(\mathbb{R}^n)$, the term on the right hand side of (\ref{eq:4}) can be made as small as we like, choosing $R$ sufficiently large. Thus we have that, chosen $R$ large enough, we can take $p$ sufficiently close to $p_S$ so that
\begin{equation}
\int_{\widetilde\Omega^+_p\cap |x|> R}|\widetilde u_p|^{\frac{n(p-1)}{2}}dx\leq\epsilon\label{eq:8}.
\end{equation}
Let us now estimate the second term of (\ref{eq:33})
\begin{eqnarray}
&&\int_{\widetilde\Omega^-_p\cap|x|>R}p|\widetilde u_p|^{p-1}\widetilde\varphi_{1,p}^2dx\nonumber\\
&&\leq p\Bigg(\frac{\|u^-_p\|_{L^\infty(\Omega)}}{\|u^+_p\|_{L^\infty(\Omega)}}\Bigg)^{p-1}
\Bigg(\int_{\widetilde\Omega^-_p\cap|x|>R}\widetilde\varphi_{1,p}^2dx\Bigg)\nonumber\\
&&\leq p\Bigg(\frac{\|u^-_p\|_{L^\infty(\Omega)}}{\|u^+_p\|_{L^\infty(\Omega)}}\Bigg)^{p-1}
\xrightarrow{p\shortrightarrow p_S} 0\nonumber
\end{eqnarray}
where we used the fact that $\|\widetilde\varphi_{1,p}\|_{L^2(\mathbb{R}^n)}=1$ and
 condition (b) satisfied by our solutions.\\
Recalling that $\widetilde u_p\xrightarrow{p\shortrightarrow p_S} U$ in $C^2_{loc}(\mathbb{R}^n)$, for $R$ fixed as above and $p$ sufficiently close to $p_S$, we have
\begin{equation}
\int_{\widetilde\Omega_p\cap|x|\leq R} \big(p_S|U|^{p_S-1}-p|\widetilde u_p|^{p-1}\big)\widetilde\varphi_{1,p}^2dx\leq\epsilon.\label{eq:2}
\end{equation}
Thus (\ref{eq:step1}) follows from (\ref{eq:15})-(\ref{eq:2}).

\item[Step 2.]
Now we show that for $\epsilon>0$ we have
\begin{equation}
\widetilde\lambda_{1,p}\leq\lambda_1^* +\epsilon\quad \mbox{ for $p$ sufficiently close to $p_S$.}\label{eq:step2}
\end{equation}
Let us consider a regular cut-off  function $\psi_R(x)=\psi_R(r)$, for $R>0$, such that
\begin{itemize}
\item[-] $0\leq\psi_R\leq 1$ and $\psi_R(r)=1$ for $r\leq R$, $\psi_R(r)=0$
for $r\geq 2R$,
\item[-] $|\triangledown \psi_R|\leq \frac{2}{R}$
\end{itemize}

\

and let us set
\begin{equation*}
w_R:=\frac{\psi_R\varphi_1^*}{\|\psi_R\varphi_1^*\|_{L^2(\mathbb{R}^n)}}.
\end{equation*}
Thus
\begin{eqnarray}
\widetilde\lambda_{1,p}&\leq& \int_{\mathbb{R}^n}|\triangledown w_R|^2-
p|\widetilde u_p|^{p-1}w_R^2dx\label{eq:11}\\
&=&\int_{\mathbb{R}^n}|\triangledown w_R|^2-p_S|U|^{p_S-1}w_R^2dx\nonumber\\
&+&\int_{\mathbb{R}^n} (p_S|U|^{p_S-1}-p|\widetilde u_p|^{p-1})w_R^2dx.\nonumber
\end{eqnarray}
It is easy to see that $w_R\shortrightarrow\varphi_1^*$ in $H^1(\mathbb{R}^n)$ as $R\shortrightarrow\infty$. Therefore, by (\ref{eq:inf}), we have that given $\epsilon>0$ we can fix $R>0$ such that
\begin{equation}\nonumber
\int_{\mathbb{R}^n}|\triangledown w_R|^2-p_S|U|^{p_S-1}w_R^2dx\leq \lambda_1^*+\epsilon.
\end{equation}
For such a fixed value of $R$, arguing as in Step 1, we obtain that
\begin{equation}
\int_{\mathbb{R}^n} (p_S|U|^{p_S-1}-p|\widetilde u_p|^{p-1})w_R^2dx\leq\epsilon\label{eq:10}
\end{equation}
for $p$ close enough to $p_S$. Then (\ref{eq:step2}) follows from (\ref{eq:11})-(\ref{eq:10}).
\\
\noindent
By \eqref{eq:step1} and \eqref{eq:step2} we deduce \eqref{gfgfgbvbvbncncnmxmx}.
\end{description}
\endproof
\begin{cor}
$\widetilde\varphi_{1,p}$ strongly converges to $\varphi_1^*$ in $L^2(\mathbb{R}^n)$.
\end{cor}
\proof
By the definition of $\widetilde\lambda_{1,p}$, and what is stated in Theorem \ref{teo:lambda}, we have
\begin{equation}\nonumber
\int_{\widetilde\Omega_{p}}|\triangledown \widetilde\varphi_{1,p}|^2- p|U|^{p-1} \widetilde\varphi_{1,p}^2dx=\widetilde\lambda_{1,p}\shortrightarrow\lambda_1^*
\qquad\mbox{ as }p\shortrightarrow p_S.
\end{equation}
This implies that $\widetilde\varphi_{1,p}$ is a minimizing sequence for (\ref{eq:inf}), and so the assertion follows by Proposition \ref{prop:spectral} (see also Remark \ref{kdhfsfhkdhkdhgkdhg}).
\endproof

\section{Proof of Theorem \ref{teo:integrale}}\label{kfkhfkshfkshfhsfhsfhksasasa}
We now proceed proving Theorem \ref{teo:integrale}.
\begin{proof}[Proof of Theorem \ref{teo:integrale}]
Using $\varphi_{1,p}\in H_0^1(\Omega)$ as a test function in (\ref{eq:probl}) we have
\begin{equation}
\int_{\Omega}\triangledown u_p\cdot\triangledown\varphi_{1,p}dx=
\int_{\Omega} |u_p|^{p-1}u_p \varphi_{1,p}dx,\label{eq:12}
\end{equation}
while using $u_p$ as a test function in (\ref{eq:varphi}) we obtain
\begin{equation}
\int_{\Omega}\triangledown u_p\cdot\triangledown\varphi_{1,p}dx=
\int_{\Omega} p|u_p|^{p-1}u_p \varphi_{1,p}dx+\lambda_{1,p}\int_{\Omega}u_p \varphi_{1,p}dx.\label{eq:13}
\end{equation}
Subtracting (\ref{eq:12}) from (\ref{eq:13}) we get
\begin{equation*}
-\frac{p-1}{\lambda_{1,p}}\int_{\Omega}|u_p|^{p-1}u_p \varphi_{1,p}dx=\int_{\Omega}u_p \varphi_{1,p}dx.
\end{equation*}
Taking into account that $\lambda_{1,p}$ is negative, we have that, to determine the sign of $\int_{\Omega}u_p \varphi_{1,p}dx$,
we can study the sign of
\begin{equation}
\int_{\Omega}|u_p|^{p-1}u_p \varphi_{1,p}dx.\label{eq:21}
\end{equation}
For convenience we consider
\begin{equation*}
M_p^{(\frac{p-1}{2})\frac{n}{2}-p}\int_{\Omega}|u_p|^{p-1}u_p \varphi_{1,p}dx
\end{equation*}
which has the same sign of (\ref{eq:21}).
Now we prove that
\begin{equation}\label{eq:limite}
\begin{split}
M_p^{(\frac{p-1}{2})\frac{n}{2}-p}\int_{\Omega}|u_p|^{p-1}u_p \varphi_{1,p}dx&\overset{p\shortrightarrow p_S}{\longrightarrow}\int_{\mathbb{R}^n}|U|^{p_S-1}U\varphi_1^*dx\\
&\quad=\int_{\mathbb{R}^n}U^{p_S}\varphi_1^*dx\,.
\end{split}
\end{equation}
Since the term on the right hand side of (\ref{eq:limite}) is positive, this will leads to the assertion of Theorem \ref{teo:integrale}.\\
By a simple change of variables it follows that:
\begin{equation}\label{eq:16}
\begin{split}
&\Bigg|M_p^{(\frac{p-1}{2})\frac{n}{2}-p}\int_{\Omega}|u_p|^{p-1}u_p\varphi_{1,p}dx-\int_{\mathbb{R}^n}|U|^{p_S-1}U\varphi_1^*dx\Bigg|\\
&=\Bigg|\int_{\widetilde\Omega_p}|\widetilde u_p|^{p-1}\widetilde u_p\widetilde\varphi_{1,p}dx-
\int_{\mathbb{R}^n}|U|^{p_S-1}U\varphi_1^*dx\Bigg|.
\end{split}
\end{equation}
We take $\epsilon>0$ and choose $R>0$ such that
\begin{equation}\nonumber
\int_{|x|> R}|U|^{p_S-1}U\varphi_1^*dx=
\int_{|x|> R}U^{p_S}\varphi_1^*dx\leq\epsilon,
\end{equation}
this is possible arguing as we did  in the proof of (\ref{eq:15}).\\
We rewrite (\ref{eq:16}) in the following way
\begin{eqnarray}
&&\Bigg|\int_{\widetilde\Omega_p\cap|x|>R}|\widetilde u_p|^{p-1}\widetilde u_p\widetilde\varphi_{1,p}dx+
\int_{\widetilde\Omega_p\cap|x|\leq R}|\widetilde u_p|^{p-1}\widetilde u_p\widetilde\varphi_{1,p}dx\nonumber\\
&-&\int_{|x|\leq R}|U|^{p_S-1}U\varphi_1^*dx
-\int_{|x|>R}|U|^{p_S-1}U\varphi_1^*dx\Bigg|\nonumber\\
&\leq&\Bigg|\int_{\widetilde\Omega_p\cap|x|>R}|\widetilde u_p|^{p-1}\widetilde u_p\widetilde\varphi_{1,p}dx\Bigg|
+\Bigg|\int_{|x|>R}|U|^{p_S-1}U\varphi_1^*dx\Bigg|\nonumber\\
&+&\Bigg|\int_{\widetilde\Omega_p\cap|x|\leq R}|\widetilde u_p|^{p-1}\widetilde u_p\widetilde\varphi_{1,p}dx
-\int_{|x|\leq R}|U|^{p_S-1}U\varphi_1^*dx\Bigg|\,.\nonumber
\end{eqnarray}
Now we analyze each term in the previous inequality. 
 Splitting the integral on $\widetilde\Omega_p^+$ and on $\widetilde\Omega_p^-$ (see (\ref{eq:omega+}) and  (\ref{eq:omega-}) for the definitions of such sets) we have
\begin{eqnarray}
&&\Bigg|\int_{\widetilde\Omega_p\cap|x|>R}|\widetilde u_p|^{p-1}\widetilde u_p\widetilde\varphi_{1,p}dx\Bigg|\label{eq:35}\\
&&\leq \int_{\widetilde\Omega^+_p\cap|x|>R}|\widetilde u_p|^{p}\widetilde\varphi_{1,p}dx+
\int_{\widetilde\Omega^-_p\cap|x|>R}|\widetilde u_p|^{p}\widetilde\varphi_{1,p}dx.\nonumber
\end{eqnarray}
As for the first term of (\ref{eq:35}) we have
\begin{eqnarray}
\int_{\widetilde\Omega^+_p\cap|x|>R}|\widetilde u_p|^{p}\widetilde\varphi_{1,p}dx
&\leq& \Big(\int_{\widetilde\Omega^+_p\cap|x|>R}|\widetilde u_p|^{\frac{2np}{n+2}}dx\Big)^{\frac{n+2}{2n}}
\Big(\int_{\widetilde\Omega^+_p\cap|x|>R}\widetilde\varphi_{1,p}^{\frac{2n}{n-2}}dx\Big)^{\frac{n-2}{2n}}\nonumber\\
&\leq& \Big(\int_{\widetilde\Omega^+_p\cap|x|>R}|\widetilde u_p|^{\frac{2np}{n+2}}dx\Big)^{\frac{n+2}{2n}}
\|\widetilde\varphi_{1,p}\|_{L^{\frac{2n}{n-2}}(\mathbb{R}^n)}\nonumber\\
&\leq& C_4 \Big(\int_{\widetilde\Omega^+_p\cap|x|>R}|\widetilde u_p|^{\frac{2np}{n+2}}dx\Big)^{\frac{n+2}{2n}}\nonumber
\end{eqnarray}
where we have used H\"{o}lder's inequality (with exponents $\frac{2n}{n+2}$ and $\frac{2n}{n-2}$) for the first estimate and the fact that, as a consequence of Lemma \ref{lemma:bounded}, $\widetilde\varphi_{1,p}$ is bounded in $L^{\frac{2n}{n-2}}(\mathbb{R}^n)$.\\
Thus, with the same argument used to obtain (\ref{eq:8}), we can state that, for every $\epsilon >0$, having chosen $R$ large enough and taking $p$ close enough to $p_S$, we have
\begin{equation}\nonumber
C_4 \Big(\int_{\widetilde\Omega^+_p\cap|x|>R}|\widetilde u_p|^{\frac{2np}{n+2}}dx\Big)^{\frac{n+2}{2n}}<\epsilon.
\end{equation}
Next we estimate the second term of (\ref{eq:35}). We have
\begin{eqnarray}
&&\int_{\widetilde\Omega^-_p\cap|x|>R}|\widetilde u_p|^{p}\widetilde\varphi_{1,p}dx\nonumber \\
&&\leq\Big(\int_{\widetilde\Omega^-_p\cap|x|>R}|\widetilde u_p|^{2p}dx\Big)^{\frac{1}{2}}\Big(\int_{\widetilde\Omega^-_p\cap|x|>R}\widetilde\varphi_{1,p}^2dx\Big)^{\frac{1}{2}}\nonumber\\
&&=\Big(\int_{\widetilde\Omega^-_p\cap|x|>R}|\widetilde u_p|^{2p-\frac{2n}{n-2}}|\widetilde u_p|^{\frac{2n}{n-2}}dx\Big)^{\frac{1}{2}}\Big(\int_{\widetilde\Omega^-_p\cap|x|>R}\widetilde\varphi_{1,p}^2dx\Big)^{\frac{1}{2}}\nonumber\\
&&\leq\Big(\frac{\|u^-_p\|_{L^\infty(\Omega)}}{\|u^+_p\|_{L^\infty(\Omega)}}\Big)^{p-\frac{n}{n-2}}
\Big(\int_{\widetilde\Omega^-_p\cap|x|>R}|\widetilde u_p|^{\frac{2n}{n-2}}dx\Big)^{\frac{1}{2}}\nonumber\\
&&\leq C_5\Big(\frac{\|u^-_p\|_{L^\infty(\Omega)}}{\|u^+_p\|_{L^\infty(\Omega)}}\Big)^{p-\frac{n}{n-2}}
\xrightarrow{p\shortrightarrow p_S} 0\nonumber
\end{eqnarray}
where we have used H\"{o}lder's inequality (with exponent 2) for the first estimate, the fact that $\|\widetilde\varphi_{1,p}\|_{L^2(\mathbb{R}^n)}=1$ for the second and condition (b) satisfied by our solution. Note in particular that, for $p$ close to $p_S$, we may and do assume that $p>\frac{n}{n-2}$.

\noindent Moreover, recalling once again that $\widetilde u_p\xrightarrow{p\shortrightarrow p_S} U$ in $C^2_{loc}(\mathbb{R}^n)$, we deduce that:
\begin{eqnarray}
\Bigg|\int_{\widetilde\Omega_p\cap|x|\leq R}|\widetilde u_p|^{p-1}\widetilde u_p\widetilde\varphi_{1,p}dx-\int_{|x|\leq R}|U|^{p_S-1}U\varphi_1^*dx\Bigg|<\epsilon\label{eq:24},
\end{eqnarray}
for $R$ fixed as above and $p$ sufficiently close to $p_S$.
 
\noindent Finally, for $R$ sufficiently large, the term $$\int_{|x|>R}|U|^{p_S-1}U\varphi_1^*dx$$ can be made arbitrary small  since  $U\in L^{\frac{2n}{n-2}}(\mathbb{R}^n)$ and $\varphi_1^*$ is bounded.

Thus (\ref{eq:16})-(\ref{eq:24}) and the arbitrary choice  of $\epsilon$ imply (\ref{eq:limite}) concluding the proof.\\
\end{proof}

\pagestyle{plain}

\end{document}